\documentclass{article}
\usepackage[utf8]{inputenc}
\usepackage{hyperref}
\usepackage{mathtools}

\def\MR#1{MR \href{http://www.ams.org/mathscinet-getitem?mr=#1}{#1}}
\def\doi#1{DOI: \href{https://doi.org/-getitem?doi=#1}{#1}}
\usepackage{amssymb}
\usepackage{amsmath}
\usepackage{amsthm}
\newtheorem{lemma}{Lemma}
\newtheorem{theorem}{Theorem}
\newtheorem{concl}{Corollary}
\newtheorem{prop}{Proposition}
\newtheorem{note}{Remark}
\date{}
\author{E.Yu.~Lerner}
\title{Matroid variant of Matiyasevich formula and its application}
\begin{document}

\maketitle

\begin{abstract} 
In 1977, Yu.~V.~Matiyasevich proposed a formula expressing the chromatic polynomial of an arbitrary graph as a linear combination of flow polynomials of subgraphs of the original graph. In this paper, we prove that this representation is a particular case of one (easily verifiable) formula, namely, the representation of the characteristic polynomial of an arbitrary matroid as a linear combination of characteristic polynomials of dual matroids. As an application, we represent the flow polynomial of a complete graph with $n$ vertices as the sum of elementary terms with respect to all partitions of positive integer $n$. Since the growth rate of the number of all partitions is less than exponential, this technique allows us to evaluate the flow polynomial for values of $n\approx 50$. We also get an explicit expression for the characteristic polynomial of the matroid dual to the matroid of the projective geometry over a finite field. We prove, in particular, that major coefficients of all these polynomials coincide with the beginning of the row in the Pascal triangle, whose number equals the quantity of elements in the corresponding matroid.  At the end part of the paper, we consider one more approach, which allows us to obtain the same results of application of our main theoren by using properties of the Tutte polynomial and the classical Rota formula for coefficients of the characteristic polynomial of a matroid. In addition, we describe the connection between the matroid variant of the Matiyasevich formula and convolution formulas for Tutte polynomials.
\end{abstract}



\section{The structure of the paper, the main theorem}
\subsection{Motivation and the paper structure}

The main goal of this paper is to provide a very simple derivation of the formula that expresses the characteristic polynomial of an arbitrary matroid as a linear combination of characteristic polynomials of dual matroids and to consider applications of this formula. As particular cases, this representation includes the Matiyasevich
formula~\cite{mat}, which expresses the flow polynomial of an arbitrary graph as a combination of chromatic polynomials of its subgraphs, as well as its dual formula, which uses factor graphs. We can treat both of these formulas as corollaries of one well-known correlation between vacuum Feynman amplitudes over
a residue ring~\cite{tmp,lob}.
With the help of Feynman amplitudes one can also represent a flow polynomial
as the sum of Legendre symbols~\cite{ejc}. However, in this case, everything is much simpler. For deriving a generalized variant of the Matiyasevich formula in a more convenient matroid language, we need only elementary combinatorics.

In spite of the simplicity of derivation, the new formula appears to be useful for applications. In particular, the standard approach to calculating the flow polynomial of the complete graph $K_n$ does not allow one to express it immediately in a simple form. At the same time, the expression for the chromatic polynomial of the complete graph (and its factor graphs) is trivial. Therefore, for calculating the flow polynomial of the complete graph it is natural to use the approach that implies its expressing in terms of the chromatic one. As a result, one needs only to calculate the sum with respect to all partitions of a positive integer. According to the Hardy--Ramanujan formula, the number of such partitions does no grow so fast; this, in particular, allows one to easily calculate the flow polynomial of a complete graph for dimensions, to which the usual approach is inapplicable.

An analogous situation occurs in calculating the characteristic polynomial of the matroid which is dual
to $PG(n-1,q)$, i.\,e., the matroid of an $(n-1)$-dimensional projective space over a finite
field~$\mathbb F_q$.
The characteristic polynomial of the matroid $PG(n-1,q)$ has roots $1,q,q^2,\ldots,q^{n-1}$
and therefore the expression for it is evident, as distinct from the characteristic polynomial of the dual matroid.
In factorization of the matroid $PG(n-1,q)$ in a nontrivial case there occur matroids $PG(k-1,q)$,
$k\leq n$; their number equals the $q$-binomial
coefficient ${n \choose k}_q$.
As a result, we obtain an explicit expression
for the characteristic polynomial of the matroid dual to the matroid $PG(n-1,q)$.

We describe these results in the section~2, together with the Matiyasevich formula (deduced as a corollary of the main theorem) and the main results obtained in~\cite{lob}.
We observe that major coefficients of all characteristic polynomials of matroids considered in Section~2 (matroids dual to matroids of projective geometries, uniform matroids,
bond matroids of complete graphs, i.\,e., their flow polynomials) coincide (accurate to alternating signs) with beginnings of rows of the Pascal triangle.
In Section~3, we study the connection of the main theorem with convolution formulas
for Tutte polynomials obtained earlier in papers~\cite{reiner3} and~\cite{kung2}. In this section, we tried to thoroughly systemize results obtained for Tutte polynomials in many papers; their more concrete statements with certain additions have allowed us to obtain formulas given in sections~1 and~2. The author of this paper is most impressed by the paper~\cite{tutte1967} published by Tutte in~1967. As appeared, the formulas for Tutte polynomials of complete graphs given in the mentioned paper are the same as correlations for flow polynomials obtained by us in Section~2 of this paper.
Moreover, Tutte considers the exponential generating function for Tutte polynomials of complete graphs. This function allows one to obtain a short expression for the flow polynomial of a complete graph in terms of the differentiation operation.

Now we can easily implement two algorithms for calculating coefficients of the flow polynomial of the complete graph with $n$ vertices that correspond to the obtained formulas (see~\eqref{f1} and~\eqref{tutFlowDif}); these algorithms are asymptotically much faster than the standard procedure in Wolfram Mathematica.
The efficiency of the first algorithm slightly exceeds that of the second one; at the beginning of Section~3 we discuss this property in more detail.

At the end part of this paper, in Conclusion, we discuss references and further research directions.

\subsection{The main theorem}
Let us now state the main theorem. By introducing the elementary ``$\zeta$-fu\-nc\-tion’’ $\zeta_q(z)=\frac{1}{1-q^{-z}}$, $q\not\in \{0,1\}$, $z\in \mathbb{C}$, we state it in the most symmetric form. The function $\zeta_q$ satisfies the following trivial functional equations:
\begin{equation}
\zeta_q(z)=-q^z \zeta_q(-z),\qquad \zeta_q(z)-1=-\zeta_q(-z). \label{eq:funczetaq}
\end{equation}

Recall~\cite{oxley} that we understand a \textit{matroid}~$M=(E,r)$ as the set~$E$ (we consider only the case of finite sets) with the function $r\equiv r_M$, which is defined on all subsets $2^E$ of the set~$E$, takes on nonnegative integer values $\mathbb N_0$,
and possesses the following properties:\\
(1) If $A\subseteq E$, then $0 \leq r(A) \leq |A|$.\\
(2) If $A \subseteq B \subseteq  E$, then $r(A) \leq r(B )$.\\
(3) If $A$ and $B$ are subsets of $E$, then
$$
r(A \cup B ) + r(A \cap B ) \leq r(A) + r(B ).
$$
The function~$r$ is the \textit{rank function} of $M$, and $E$ is the \textit{ground set} of $M$.

\textit{The dual matroid} $M^*=(E,r^*)$ to the matroid $M=(E,r)$ is defined by the rank function $r^*$ which is defined on the same set as~$r$ and obeys the correlation
\begin{equation}
\label{eq:rstar}
r^*(A) = r(E - A) + |A| - r(E),\quad \mbox{for any $A\subseteq E$.}
\end{equation}

Let $A\subseteq E$. \textit{The restriction} of $M=(E,r)$ to $A$ is the matroid $M|A$ with the ground set $A$ and the rank function coinciding with $r$ on $2^A$. Let $C=E-A$.
\textit{The contraction} of $M$ onto $A$ is the matroid $M.A\equiv M/C$  with the ground set~$A$ and the rank function obeying the correlation
$$
r_{M.A}(B)= r_M (C \cup B) - r_M (C)\quad \mbox{for any $B\subseteq A$.}
$$
Evidently, if $B\subseteq A \subseteq E$, then 
$(M|B)|A=M|A$ and $(M.A).B=M.B$\,.In addition, as is known~\cite{oxley},
\begin{equation}
\label{MX}
(M.A)^*=M^* |A, \quad (M|A)^*=M^* .A\ .
\end{equation}

\textit{The characteristic polynomial} $\chi_M$ of the matroid~$M=(E,r)$ obeys the formula
\begin{equation}
\label{eq:chiM}
\chi_M(x) = \sum_{A\subseteq E} (-1)^{|A|} x^{r(E)-r(A)}.
\end{equation}

Let us give some more definitions, which will be necessary in the following sections of this paper. The \textit{flat} of a matroid is its inclusion-maximal subset $A$ of some fixed rank. This means that there exists no element $e\in E-A$ such that $r(A\cup\{e\})=r(A)$.
We understand a link as an element $e\in E$ such that $r(\{e\})=0$. If the rank of an analogous subset in the matroid~$M^*$ equals zero, then $e$ is called a colink.
One can easily make sure that
\begin{equation}
\label{eq:chiMwithLoops}
\mbox{if a matroid $M$ contains links, then $\chi_M(x)\equiv 0$.}
\end{equation}

\begin{theorem}
\label{main}
Let $M=(E,r)$ be an arbitrary matroid on the ground set~$E$ with the rank function~$r$. The following formulas are valid:
\label{ter:1}
\begin{eqnarray}
\label{one}
& \chi_{M^*}(q)\,\zeta_q(-1)^{|E|}=\sum\limits_{A\subseteq E} (-1)^{|E|-|A|}\,
\frac{\chi_{M|A}(q)}{q^{r(A)}}\, \zeta_q(1)^{|A|},&\\
\label{two}
&
\frac{\chi_{M^*}(q)}{q^{r^*(E)}}\zeta_q(1)^{|E|}=\sum\limits_{A\subseteq E} \zeta_q(-1)^{|E|-|A|}\,\chi_{M/A}(q).
&
\end{eqnarray}
\end{theorem}

\begin{proof}
Let $A\subseteq E$. In accordance with formula~\eqref{MX} correlation~\eqref{one} means that
$$
\chi_{M^*.A}(q)\,\zeta_q(-1)^{|A|}=\sum\limits_{B\subseteq A} (-1)^{|A|-|B|}\,
\frac{\chi_{M|A}(q)}{q^{r(B)}}
\, \zeta_q(1)^{|B|}.
$$
Therefore, applying the M\"{o}bius inversion formula (in the case under consideration it coincides with the inclusion-exclusion formula) to identity~\eqref{one}, we can rewrite the latter as follows:
\begin{equation}
\label{twozeta}
\frac{\chi_M(q)}{q^{r(E)}}\zeta_q(1)^{|E|}=\sum_{A\subseteq E} \zeta_q(-1)^{|A|}\,\chi_{M^*.A}(q).
\end{equation}
Replacing $A$ with~$E-B$ and doing the matroid~$M$ in formula~\eqref{twozeta} with $M^*$ (recall that $\left(M^*\right)^*=M$) we get formula~\eqref{two}.
Therefore, it suffices to prove only correlation~\eqref{one}.

Performing substitution~\eqref{eq:rstar} and putting $A=E-B$ in the left-hand side of formula~\eqref{one}, we turn it into the form
\begin{equation}
\label{eq:left}
\sum_{A\subseteq E} (-1)^{|E|-|A|}\,q^{|A|-r(A)}  \zeta_q(-1)^{|E|}.
\end{equation}

We can write the right-hand side of formula~\eqref{one} as follows:
\begin{eqnarray}
& &\sum_{A\subseteq E} q^{-r(A)} \zeta_q(1)^{|A|} 
\sum_{B: A\subseteq B\subseteq E} \zeta_q(1)^{|B|-|A|} (-1)^{|E|-|B|}= \nonumber \\
&=&\sum_{A\subseteq E} q^{-r(A)}
(-\zeta_q(1))^{|A|}
(\zeta_q(1)-1)^{|E|-|A|}.
\label{eq:finalproof}
\end{eqnarray} 
Substituting correlations~\eqref{eq:funczetaq} in~\eqref{eq:finalproof}, we conclude that formula~\eqref{eq:finalproof} is equivalent to~\eqref{eq:left}.
\end{proof}

Note that the first of correlations~\eqref{eq:funczetaq} and the equality $r^*(E)=|E|-r(E)$ allow us to rewrite formula~\eqref{two} in a form, which is more convenient for calculations.

\begin{concl}
\label{concl1}
The following formula is valid:
\begin{equation}
\label{finaltwo}
\chi_{M^*}(x) =\frac{(-1)^{|E|}}{x^{r(E)}}\sum\limits_{A\subseteq E}  (1-x)^{|A|}\,\chi_{M/A}(x).
\end{equation}
\end{concl}

Just this corollary we use in most applications.

\section{Applications of Theorem~\ref{main} and Corollary~\ref{concl1}}
\subsection{Applications in the case of the graphic matroid}

Let~$G=(V,E)$ be an undirected graph with the vertex set~$V$ and the edge set~$E$; $q\in \mathbb N$.
The number of ways of coloring vertices of the graph in~$q$ colors so as to make adjacent vertices of the graph have various colors is defined by the \textit{chromatic polynomial} of the graph $P_G(q)$.
The polynomial $P_G(q)$ is a multiple of~$q^{c(G)}$, where $c(G)$ is the number of connectivity components of the graph~$G$.
Its normalized version $P'_G(q)=P_G(q)/q^{c(G)}$ coincides with the characteristic polynomial of the so-called \textit{graphic} matroid $M_G=(E,r)$.
The rank function $M_G$ obeys the formula $r(A)=|V(A)|-c(A)$, where $|V(A)|$ and $c(A)$ are, correspondingly, the number of vertices and the number of connectivity components of the subgraph induced by the subset of edges~$A$, $A\subseteq E$.

The characteristic polynomial of the dual (\textit{bond}) matroid $M^*_G$ is the so-called \textit{flow polynomial} $F_G$~\cite{rota}. Let us fix some orientation of edges in~$E$ and denote by the symbol $\mathbb Z_q$ the additive group of integers modulo $q$.
In what follows, we understand a nonzero flow as a correspondence between each directed edge $e\in E$ and a nonzero element $k_e$ in $\mathbb Z_q$ such that for each vertex~$v\in V$ the sum $\sum k_e$ for all $e$ that enter in $v$ coincides with the analogous sum for all $e$ outgoing~$v$. The number of everywhere nonzero flows is independent of the orientation of edges; it equals the value of the flow polynomial $F_G(q)$.

In the case, when $M=M_G$, formula~\eqref{twozeta} takes the form
$$
\frac{P_G (q)}{q^{|V|}}\zeta _q (1)^{|E|}
=\sum_{A \subseteq E} \zeta _q(-1)^{|A|} F_{H(A)} (q),
$$
where $H(A)$ is the subgraph of the graph~$G$ induced by the subset~$A$ of edges of this graph.
The latter correlation coincides with the Matiyasevich formula~\cite{mat} accurate to denotations.

Correlation~\eqref{one} implies the formula that is inverse with respect to the above one, namely,
$$
F_G (q) {\zeta _q (-1)}^{|E|}  =
\sum_{A\subseteq E}  (q- 1)^{|E|-|A|}
\frac{P_{H(A)} (q)}{q^{|V(A)|}} \zeta_q (1)^{|A|},
$$
where $|V(A)|$ is the number of vertices of the graph~$H(A)$. See this representation of the flow polynomial as a linear combination of chromatic polynomials at the end part of the paper~\cite{lob} (see also formula~(2.5) in~\cite{kazakov}).

We can obtain a more interesting property, substituting the matroid $M=M_G$ in correlation~\eqref{finaltwo}.
We can express matroids $M/A$ in terms of factor graphs of the graph~$G$ with respect to subgraphs $H(A)$. Moreover, it is evident that if a factor graph $G/H$ contains links then $P_{G/H}(q)\equiv 0$. As a result, we obtain a slightly modified statement of theorem~2 in the paper~\cite{lob}.

\begin{theorem}
Assume that the set of vertices $V$ is divided into subsets $V_1,\ldots ,V_m$, where $m$ is an arbitrary number,
but $V_i$ are such that $m$ subgraphs formed by all edges in~$E$, whose both endpoints belong to one and the same subset $V_i$, are connected. Denote by~$A$ the set of all these edges and do by $H(A)$ the union of all subgraphs, correspondingly. Then the following formula is valid:
\begin{equation}
\label{eq:th2}
F_G(x)=\frac{(-1)^{|E|}}{x^{|V|}} \sum_{H(A)} (1-x)^{|A|} P_{G/H(A)}(x);
\end{equation}
here the sum is calculated over all subgraphs of~$H$ (all such partitions of the set~$V$).
\end{theorem}

Here, similarly to formula~\eqref{finaltwo} given above, and in what follows, we use the ``real'' variable $x$ instead of~$q$ (we will need the symbol $q$ later for the standard denotation of the number of elements in a finite field, though earlier it was convenient for us to use this variable as a subscript for the ``zeta-function''). Let us now apply formula~\eqref{eq:th2} in the case, when the graph~$G$ is $K_n$, i.\,e., a complete graph with $n$ vertices. Let us first recall some assertions of the partition theory~\cite{Andrews}.

A \textit{partition} $\lambda$ of a positive integer~$n$ is its representation as the sum of positive integer terms (two sums, which differ in the order of addends, are considered as equal ones). The total number of partitions $p(n)$ satisfies the well-known formula derived by Hardy and Ramanujan (and refined by Rademacher). The asymptotics for $p(n)$ is subexponential, in particular, $p(50)$ is only $204\,226$. At the same time, it is impossible to calculate (using a computer) $F_{K_n}(x)$ with $n=50$ with the help of standard techniques (which are not specific for complete graphs), because the computation time grows faster than $\exp(n)$. Formula~\eqref{eq:th2} allows us to reduce the calculation of $F_{K_n}(x)$ to summing $p(n)$ terms and to solve the problem of calculating $F_{K_n}(x)$ for $n\sim 50$.

We can represent the partition $\lambda$ of the number~$n$ in the form $(\lambda_1,\ldots,\lambda_m)$, where
$$\lambda_1\leq \lambda_2\leq\ldots\leq \lambda_m, \quad \sum\nolimits_{i=1}^m \lambda_i=n,$$
or in the form $(1^{n_1}2^{n_2}3^{n_3},\ldots)$, where $n_i$ is the number of values $i$ among the collection $(\lambda_1,\ldots,\lambda_m)$, correspondingly, $\sum_{i} i\, n_i=n$.
The number of partition elements $m$ is called the \textit{length of the partition}; we denote it as $\ell(\lambda)$. Let
\begin{eqnarray*}
\lambda!&=&\lambda_1!\lambda_2!\lambda_3!\ldots=1!^{n_1}2!^{n_2}3!^{n_3}\ldots,\\
|\lambda|!&=&n_1!n_2!n_3!\ldots,\qquad f(\lambda)=\lambda!\times |\lambda|! .
\end{eqnarray*}
The number $f(\lambda)$ is called the 
\textit{Fa\'a di Bruno coefficient} 
of the partition $\lambda$. We need the following assertion (see~\cite[theorem~13.2]{Andrews}) (one can prove it rather easily).

\begin{lemma}
\label{flambda}
The number of ways for dividing a set of $n$ distinct elements into parts, whose powers are defined by the partition $\lambda$, equals $f(\lambda)$.
\end{lemma}

\begin{theorem}
\label{Ffullgraph}
Let $s(\lambda)=\sum_{i=1}^{\ell(\lambda)} \lambda_i(\lambda_i-1)/2$.
The following formula is valid:
\begin{equation}
\label{f1}
 F_{K_n}(x) = \frac{(-1)^\frac{n(n-1)}{2}}{x^n} \sum_{\lambda} f(\lambda)
 (1-x)^{s(\lambda)} x(x-1)\ldots(x-\ell(\lambda)+1);
\end{equation}
here the sum is calculated over all partitions $\lambda$ of the number~$n$.
\end{theorem}
\begin{proof}
Let us apply formula~\eqref{eq:th2} for $G=K_n$. All terms in the right-hand side for various representations of $V$ as the union of disjoint parts $V_1,\ldots , V_m$ are the same, provided that their powers form one and the same partition $\lambda$ of the number~$n$. By Lemma~\ref{flambda} the quantity of such representations equals $f(\lambda)$. The number of edges in the subgraph~$H$ such that each their connectivity component represents a complete subgraph with $\lambda_i$ vertices obeys the evident formula
$$\sum\nolimits_{i=1}^{\ell(\lambda)}\frac{\lambda_i(\lambda_i-1)}{2};\quad \mbox{and}\ P_{K_m}(x)=x(x-1)\ldots(x-m+1).$$
\end{proof}
Formula~\eqref{f1} allows us to easily calculate, for example, the value $F_{K_{50}}(x)$.

\begin{concl}
\label{concl2}
Let $n>2$, $N=n(n-1)/2$ and $N'=(n-1)(n-2)/2$. The major $n-2$ powers of the flow polynomial $F_{K_n}(x)$ take the form
$$
x^{N'}-N x^{N'-1}+ {N \choose 2} x^{N'-2} -\ldots+(-1)^{n-3} {N \choose n-2} x^{N'- n+ 3}.
$$
\end{concl}
\begin{proof}
Consider the function $g(\lambda)=s(\lambda)+\ell(\lambda)$, where $\lambda$ is a partition of a fixed number. With $\ell(\lambda)=2$ the value of the function~$g$ decreases as so does the value $|\lambda_1-\lambda_2|$. Moreover, with $\ell(\lambda)=1$ its value is greater than that with $\ell(\lambda)=2$.
In view of this fact we conclude that with $\ell(\lambda)\not\in \{1,2\}$ we can increase the value of $g(\lambda)$ by ``merging together’’ two elements of the partition.
As a result, the largest value of $g(\lambda)$ is attained with $\lambda'$ such that $\ell(\lambda')=1$, while the second largest one is attained with 
$\lambda''=(1,n-1)$, $\ell(\lambda'')=2$.

By Theorem~\ref{Ffullgraph} the value $(g(\lambda)-n)$ is the power of the term that corresponds to the partition~$\lambda$ in sum~\eqref{f1}. Using calculations performed in the previous paragraph, we conclude that coefficients of $F_{K_n}(x)$ coincide with coefficients at major powers $(x-1)^N$, provided that the number of major powers equals
$$
g(\lambda')-g(\lambda'')=(N+1)-(N'+2)=n-2.
$$
Here the power of the polynomial $F_{K_n}(x)$ equals $g(\lambda')-n=N+1-n=N'$.
\end{proof}

As an illustration, let us consider the result of the calculation of the flow polynomial of a complete graph with 5 vertices:
$$
F_{K_5}(x)=\mathbf{1}\, x^6-\mathbf{10}\, x^5+\mathbf{45}\, x^4-115 x^3+175 x^2-147 x+51.
$$

\subsection{Applications in the non-graphic case}
Let us now discuss the application of the main theorem in cases when the matroid~$M$ is not, generally speaking, a graphic one.
Recall that~(\cite{oxley}) the uniform matroid $U_{m,n}$, $m\leq n$,
$m,n\in {\mathbb N}_0$, is defined on the ground set $E$, $|E|=n$, by the rank function $r$ as follows:
$$
\mbox{with $A\in 2^E$ we get the equality $r(A)=\min\{m,|A|\}$}.
$$
Note that $U^*_{m,n}=U_{n-m,n}$. One can easily verify that for $A\subseteq E$, $|A|=k$,
\begin{equation}
\label{umnk}
U_{m,n}|A=U_{\min\{m,k\},k}\quad\text{and}\quad U_{m,n}.A=U_{\max\{0,m-n+k\},k}.
\end{equation}
We directly calculate the characteristic polynomial~\eqref{eq:chiM} of the matroid $U_{m,n}$: $\chi_{U_{0,0}}=1$, in the rest cases (with all other values $m\leq n$) the following formula is valid:
$$
\chi_{U_{m,n}}(x)=\sum_{i=0}^{m-1} (-1)^i {n \choose i} (x^{m-i} - 1).
$$

In this case, sums over subsets of $E$ of power~$k$ in formulas in Theorem~\ref{main} and Corollary~\ref{concl1} give ${n \choose k}$ equal terms. As a result, we get some identities for binomial coefficients. In the next section, we consider a more elegant and at the same time elementary generalization of these identities (see Remark~\ref{concl3}).

Let us thoroughly study the application of formula~\eqref{finaltwo} in the case of finite projective geometries, i.\,e., projective spaces over a finite field~${\mathbb F}_q$. Recall that elements of the ground set~$E$ of the matroid~$PG(n-1,q)$ are all distinct points of the projective space of dimension~$n-1$ over the field~${\mathbb F}_q$~(see details in~\cite[Ch.~6]{oxley}). The rank of the subset~$A$ of the set~$E$ is the dimension of the linear space spanned by~$A$, in particular $r(E)=n$.
One can easily make sure that $|E|=\frac{q^n-1}{q-1}={n \choose 1}_q$; here the Gaussian binomial coefficient ${n \choose k}_q$, $n,k\in \mathbb N_0$, obeys the formula
$$
{n \choose k}_q=
\frac{(q^n-1)(q^{n-1}-1)\cdots (q^{n-k+1}-1)}{(q^k-1)(q^{k-1}-1)\cdots (q-1)}.
$$
The following assertion is a generalization of the remark about the number of elements in the ground set~(\cite[Proposition 6.1.4]{oxley}):
\begin{equation}
\label{eq:PGrankk}
\mbox{The number of rank-$k$ flats in $PG(n-1,q)$ equals ${n \choose k}_q$.}
\end{equation}
In addition, it is well known that (see~\cite[Proposition 7.5.3]{zasl} and references to fundamental works by Richard Stanley mentioned there) the characteristic polynomial of the matroid~$PG(n-1,q)$ takes the form
\begin{equation}
\label{eq:chiPG}
\chi_{PG(n-1,q)}(x)=(x-1)(x-q)(x-q^2)\cdots (x-q^{n-1}).
\end{equation}

The factor space of a linear space of dimension~$n$ with respect to its subspace of dimension~$k$, $k<n$, is its subspace of dimension~$n-k$ orthogonal to the subspace, which we use for the factorization.
In terms of matroids, this means that if $A, A\subseteq E$, is the rank-$k$ flat of $PG(n-1,q)$, then the matroid $PG(n-1,q)/A$ is isomorphic to $PG(n-k-1,q)$.
But if $A$ is not a flat of $PG(n-1,q)$, then the matroid $PG(n-1,q)/A$ contains links.
Therefore, by using formulas~\eqref{eq:chiMwithLoops}, \eqref{eq:PGrankk}, and~\eqref{eq:chiPG}, we conclude that Corollary~\ref{concl1} in the case under consideration can be formulated as follows:
\begin{theorem}
\label{chiPGdual}
The characteristic polynomial of the matroid dual to $PG(n-1,q)$ obeys the formula
$$
\chi_{PG(n-1,q)^*}(x) =\frac{(-1)^{\frac{q^n-1}{q-1}}}{x^n}\sum_{k=0}^n {n \choose k}_q\ (1-x)^{\frac{q^k-1}{q-1}}\,\prod_{i=0}^{n-k-1}(x-q^i).
$$
\end{theorem}

Here the empty product equals one. Correspondingly, the latter term in the sum equals $(1-x)^{\frac{q^n-1}{q-1}}$, while the second largest one is the polynomial of~$x$ of the power $\frac{q^{n-1}-1}{q-1}+1$. Hence, we get the following analog of Corollary~\ref{concl2}:
\begin{concl}
\label{concl5}
Let $N=\frac{q^n-1}{q-1}$, $N'=N-n$, $k=q^{n-1}-1$. 
The $k$ major powers of the polynomial $\chi_{PG(n-1,q)^*}(x)$ take the form
$$
x^{N'}-N x^{N'-1}+ {N \choose 2} x^{N'-2}  
-\ldots+(-1)^{k-1} {N \choose k} x^{N'-k+ 1}.
$$
\end{concl}

For example, the matroid dual to the Fano matroid $F_7=PG(2,2)$ satisfies the equality
$$
\chi_{F_7^*}(x)=\mathbf{1} \,x^4-\mathbf{7} \,x^3+\mathbf{21} \,x^2-28 x+13.
$$

\section{The Tutte polynomial and the matroid variant of the Matiyasevich formula}
\subsection{Definition of the Tutte polynomial}
In this section, we prove that one can obtain formulas given in sections~1 and~2 by refining results obtained for the Tutte polynomial.
The author of this paper is most impressed by paper~\cite{tutte1967}, where Tutte generalizes Theorem~\ref{Ffullgraph}.
Moreover, in the mentioned paper, Tutte also obtains the exponential generating function for Tutte polynomials of complete graphs. It implies that one can get the flow polynomial of the complete graph by multiple differentiations (see Corollary~\ref{concl4} below).
The mathematical formula in this case is more compact than formula~\eqref{f1} in~Theorem~\ref{Ffullgraph} (one can also obtain the latter formula as a corollary of the Tutte decomposition of the Tutte polynomial of the complete graph~$K_n$ in all partitions~$\lambda$ of the number~$n$). Note that the direct calculation of the flow polynomial of a complete graph by means of differentiation (using a computer) requires slightly more time than its calculation by decomposition in all partitions (in essence, both techniques are equivalent, they differ only in computation aspects and intermediate results storing methods; the latter are not specified in the shorter procedure which implies the use of the computation technique mentioned in Corollary~\ref{concl4}).

See fundamentals of the theory of Tutte polynomials, in particular, in~\cite{welsh}, \cite{Bryl}, and in the book~\cite{tutte-book} by Tutte (in the case of graphic matroids).
Specific properties of characteristic polynomials that are of interest for us are described in classical papers~\cite{rota} and~\cite{zasl}.
The review~\cite{mexico} has appeared to be very useful for us as a reference book of formulas for various Tutte polynomials.

Recall that the \textit{Tutte polynomial} of the matroid $M=(E,r)$ obeys the formula
$$
T_M(x,y)=\sum_{A\subseteq E} (x-1)^{r(E)-r(A)}(y-1)^{|E|-r(A)}.
$$
By definition, $T_{M^*}(x,y)=T_M(y,x)$. In addition,
\begin{equation}
\label{eq:chifromT}
\chi_M(z)=(-1)^{r(E)} T_M(1-z,0),\qquad   \chi_{M^*}(z)=(-1)^{|E|-r(E)} T_M(0,1-z).
\end{equation}
In certain (exceptional) cases one can easily calculate the Tutte polynomial. For example, in the case of a uniform matroid $U_{m,n}$, it is evident that
$$
T_{U_{m,n}}(x,y)=\sum_{i=0}^{m-1} {n \choose i}(x-1)^{m-i} + {n \choose m} + \sum_{i=m+1}^{n} {n \choose i}(y-1)^{i-m}.
$$
In a general case, for any matroid with $(x,y)\in H_1$, where $H_1$ is the so-called critical hyperbola $(x-1)(y-1)=1$, the following formula is valid:
\begin{eqnarray}
T_M(x,y(x))&=&\sum_{A\subseteq E} (x-1)^{r(E)-r(A)} (x-1)^{r(A)-|A|}=
\sum_{A\subseteq E}\frac{(x-1)^{r(E)}}{(x-1)^{|A|}}=\nonumber \\
&=&(x-1)^{r(E)}\left(1+\frac{1}{(x-1)}\right)^{|E|}=x^{|E|}(x-1)^{r(E)-|E|}.\label{eq:Tx1x}
\end{eqnarray}

If the ground set $E$, $|E|=n$, consists of $i$ links and $j$ colinks (the base case), then $T_M(x,y)=x^j y^i$, $i+j=n$. In particular, in this case, the characteristic polynomial satisfies the correlation
\begin{equation}
\label{chi_Mcolink}
\chi_M(z)=z^n\quad\mbox{if $j=n$;}\qquad \chi_M(z)=0\quad\mbox{otherwise}.
\end{equation}

Assume that an edge~$e$ is neither a link nor a colink, and let $A=E-e$ (hereinafter we identify the set consisting of one edge and the edge itself). Denote the matroid $M|A$ as $M\setminus e$. Any matroid and any $e\in E$ different from a colink satisfy the formula
$T_M(x,y)=T_{M\setminus e}(x,y)+T_{M/e}(x,y)$ (the Tutte polynomial is the Tutte--Grothendieck invariant).

The latter correlation, together with the base case, uniquely defines $T_M(x,y)$ for any~$M$.
In the case of the characteristic polynomial, the recurrent correlation for it takes the form
\begin{equation}
\label{chi_Mrec}
\chi_M(z)=\chi_{M\setminus e}(z)-\chi_{M/e}(z).
\end{equation}

\subsection{Theorem~\ref{main} as a corollary of the convolution formula for the Tutte polynomial}
Correlations~\eqref{chi_Mcolink} and~\eqref{chi_Mrec} allow us to prove formula~\eqref{one} by induction with respect to the number of edges different from colinks.
Below we discuss a more natural proof of Theorem~\ref{main} which is based on the application of the convolution formula for Tutte polynomials. This important correlation for Tutte polynomials established by W.~Kook, V.~Reiner, and D.~Atanton allows several generalizations.

\begin{prop}[\cite{reiner3}] The following formula is valid:
\label{prop1}
\begin{equation}
\label{rein1}
T_M(x,y)=\sum_{A\subseteq E} T_{M|A} (0,y) T_{M/A} (x,0).
\end{equation}
\end{prop}

For example, for uniform matroids, the following remark is valid.

\begin{note}
\label{concl3}
Let $n\in\mathbb N$, $m\in \mathbb N_0$, $m\leq n$. For the matroid $M=U_{m,n}$ identity~\eqref{rein1} takes the form
$$
T_M(x,y)=T_{M} (x,0) + T_{M} (0,y).
$$
\end{note}

\begin{proof}
Let us use formula~\eqref{umnk} and the fact that the sign alternating sum of binomial coefficients equals zero. With $k\in\mathbb N$ we get equalities $$T_{U_{k,k}}(0,y)=0;\qquad T_{U_{k,0}}(x,0)=0. $$
Therefore, with
$0<|A|\leq m$ we get the equality $T_{U_{m,n}|A}(0,y)=0$, and with $m \leq |A|<n$ we conclude that $T_{U_{m,n}/A}(x,0)=0$. Therefore, among all terms in the sum in the right-hand side of the equality \eqref{rein1}, only those ones that correspond to cases $A=\emptyset$ and $A=E$ differ from zero.
\end{proof}

Certainly, one can easily immediately verify the formula in Remark~\ref{concl3} without referring to Proposition~\ref{prop1}.

We need one more notion which is closely connected with the Tutte polynomial.
Recall that the \textit{Whitney rank generating function} $R_M(u,v)$ of the matroid $M=(E,r)$ obeys the formula
$$
R_M(u,v)=\sum_{A\subseteq E}u^{r(E)-r(A)}v^{|A|-r(A)},
$$
i.\,e.,
$R_{M}(u-1,v-1)=T_{M}(u,v)$. Evidently,
\begin{equation}
\label{chifromR}
\chi_{M}(x)=(-1)^{r(E)}R_{M}(-x,-1),\quad \chi_{M^*}(x)=(-1)^{|E|-r(E)}R_{M}(-1,-x).
\end{equation}
We can rewrite formula~\eqref{eq:Tx1x} as follows:
\begin{equation}
\label{last2}
R_M(x,1/x)=(x+1)^{|E|}x^{r(E)-|E|}.
\end{equation}

In the paper~\cite{reiner2}, V.~Reiner generalizes equality~\eqref{rein1}
to a certain two-parameter family of equalities, where formula~\eqref{rein1} is a limit correlation.
Finally, in~\cite{kung2} J.~P.~S.~Kung obtains several identities
which generalize identities established earlier and the convolution formula~\eqref{rein1} in particular.
Identity~3 in this paper takes the form
\begin{equation}
\label{eq:kung2}
R_M(\lambda\xi,xy)=
\sum_{A\subseteq E} \lambda^{r(E)-r(A)}(-y)^{|A|-r(A)}R_{M|A}(-\lambda,-x)R_{M/A}(\xi,y)
\end{equation}
(we use denotations introduced by J.~P.~S.~Kung).
Now we represent formulas~\eqref{one} and~\eqref{two} as particular cases of identity~\eqref{eq:kung2}.

\begin{proof}[Proof of Theorem~\ref{main} as a corollary of equality~\eqref{eq:kung2}]
Putting $xy=-q$ and $\lambda\xi=-1$, in accordance with~\eqref{chifromR}
in the left-hand side of equality~\eqref{eq:kung2} we obtain the desired expression
$(-1)^{|E|-r(E)}\chi_{M^*}(q)$.
For establishing~\eqref{one} it suffices to concretize $x=1$, $\lambda=q$ and, correspondingly,
$y=-q$, $\xi=-1/q$; then to perform substitutions~\eqref{chifromR} and~\eqref{last2}, and sum degrees of~$q$ and~$(-1)$. Formula~\eqref{two} can be obtained analogously, provided that $y=-1$, $\xi=-q$ and, correspondingly, $x=q$ and $\lambda=1/q$.
\end{proof}

Note that Theorem~\ref{main} without referring to Matiyasevich (the author of this paper was not aware of his works) was first stated by us in preprint~\cite{lerner2016} (see theorem~2 therein).
The proof described above is adduced in the mentioned preprint.

\subsection{Explicit formulas for the Tutte polynomials of some matroids and theorems of Section 2}
All theorems in Section~2 are more elementary corollaries of results on Tutte polynomials which were obtained earlier.
We deduce Theorem~\ref{chiPGdual} by substituting values $u=0$ and $v=1-x$ in the formula for the Tutte polynomial in the projective geometry
\begin{equation}
\label{TPG}
T_{PG(n-1,q)}(u,v) =\frac{1}{(u-1)^n}\sum_{k=0}^n {n \choose k}_q\ v^{\frac{q^k-1}{q-1}}\,\prod_{i=0}^{n-k-1}((u-1)(v-1)-q^i)
\end{equation}
and by multiplying it by $(-1)^{|E|-n}$ in accordance with correlation~\eqref{eq:chifromT}.
Expression~\eqref{TPG} is established in paper~\cite{MPHAKO}.

Another way to derive the same expression~\eqref{TPG} together with the exponential generating
function for Tutte polynomials of projective geometries over the field $\mathbb F_q$ is proposed in the paper~\cite{reinePGen}. Evidently, the value of the $n$th derivative of the exponential generating function at zero coincides with the $n$th term of the series.
Below we apply such an approach for deriving an explicit formula for the characteristic polynomial of matroid,
calculating the flow polynomial of the complete graph in a different way.
The technique used by us is based on results obtained by Tutte in~\cite{tutte1967}.

In the mentioned paper, Tutte does not immediately consider the flow polynomial, but he elegantly treats the polynomial of two variables $Q_G(u,v)$ which is closely connected with the chromatic polynomial $P_G(x)$.
One can express this dichromatic Tutte polynomial $Q_G(u,v)$ in terms of the Whitney rank function as follows:
$$
Q_G(u,v)=u^{c(G)} R_{M_G}(u,v);
$$
here $c(G)$ is the number of connectivity components of the graph~$G$.
For the dichromatic polynomial of a complete graph, Tutte has obtained an equation (\cite[equation~(16)]{tutte1967}), which in terms used in Section~2 takes the form
$$
Q_{K_n}(u,v)=\frac{1}{v^n} \sum_{\lambda} f(\lambda)
 (1+v)^{s(\lambda)} uv(uv-1)\ldots(uv-\ell(\lambda)+1).
$$
By using the second formula in~\eqref{chifromR} we obtain the statement of Theorem~\ref{Ffullgraph} as a particular case of the above correlation.

The previous correlation was strengthened by Tutte as the following formula for the exponential generating function for $Q_{K_n}(u,v) v^n$ (\cite[equation~(17)]{tutte1967}):
\begin{equation}
\label{tutgen}
\sum_{n=0}^{\infty} \frac{z^n v^n}{n!} Q_{K_n}(u,v)=\Bigl\{ \sum_{n=0}^{\infty} \frac{z^n}{n!} (1+v)^{\frac12 n(n-1)}\Bigr\}^{uv}.
\end{equation}

\begin{concl}[Corollary of the Tutte formula~\eqref{tutgen}]
The flow polynomial of a complete graph obeys the formula
\label{concl4}
\begin{equation}
\label{tutFlowDif}
F_{K_n}(x)=\frac{(-1)^{\frac12 n(n-1)}}{x^n}\left. \left( \frac{\partial^n}{\partial z^n} \Bigl\{\sum_{i=0}^n \frac{z^i}{i!}(1-x)^{\frac12 i(i-1)}\Bigr\}^x\,\right)\right|_{z=0}.
\end{equation}
\end{concl}
\begin{proof}
Substituting~\eqref{chifromR} in formula~\eqref{tutgen}, we get the correlation
$$
\sum_{i=0}^{\infty} \frac{z^i x^i (-1)^{\frac12 i(i-1)}}{i!} F_{K_i}(x)=\Bigl\{ \sum_{i=0}^{\infty} \frac{z^i}{i!} (1-x)^{\frac12 i(i-1)}\Bigr\}^x.
$$
By applying the operator $\left. \left( \frac{\partial^n}{\partial z^n}\, \cdot \, \right)\right|_{z=0}$ to both sides of the latter equality, we obtain formula~\eqref{tutFlowDif}.
\end{proof}

\subsection{Reason for binomiality of the major coefficients of characteristic polynomials}
As we have mentioned, the major coefficients of characteristic polynomials of all matroids $M^*_{K_n}$, $U_{m,n}$, and $PG(n,q)^*$ considered in Section~2 coincide (accurate to alternating signs)
with the beginning of the row of the Pascal triangle, whose number equals the quantity of elements in the corresponding matroid. We have established this property as a corollary of the matroid form of the Matiyasevich formula. The classical assertion about combinatorial sense of coefficients $\chi_M$ stated by G.~C.~Rota~\cite{rota} reveals a deeper reason for the result mentioned above. Namely, the reason is that the complete graph $K_n$, the projective geometry $PG(n,q)$, and especially the uniform matroid have ``rather thick cuts'' (cocycles). Let us discuss this property.

Recall that the set~$A$: $A\subseteq E$ of the matroid $M=(E,r)$ is said to be \textit{independent}, if $r(A)=|A|$. If, in addition, $r(A)=r(E)$, then the independent set is called a \textit{base}.
The minimal set, which is not independent, is called a \textit{cycle} of the matroid~$M$.
A cycle of the matroid $M^*$ is called a \textit{cocycle}.
Let us associate all elements $e$, $e\in E$, with distinct positive integers $l_e$.
A cycle~$C$ without an element $e'$ such that $l_{e'}=\max_{e\in C} l_e$ is called a \textit{broken circuit}.
G.~C.~Rota has obtained the following result (\cite[P.~359, Corollary~1]{rota}):

\begin{prop}
\label{prop2}
Assume that the characteristic polynomial of the matroid $M=(E,r)$ takes the form
$$\chi_M(x)=x^{r(E)}+m_1 x^{r(E)-1}+m_2 x^{r(E)-2}+\cdots+m_{r(E)}.$$
Then $(-1)^k m_k$ is the number of independent subsets~$E$ consisting of $k$ elements without broken circuits.
\end{prop}

As a result, we obtain the following assertion. If any cocycle of the matroid $M$ contains at least $t$ elements, then for any $k\leq t-2$ the absolute value $(-1)^k m'_k$ of the $k$th coefficient of the characteristic polynomial $\chi_{M^*}$
$$
\chi_{M^*}(x)=x^{r^*(E)}+m'_1 x^{r^*(E)-1}+\cdots+m'_{r^*(E)} \quad \mbox{(here $r^*(E)=|E|-r(E)$)}
$$
coincides with ${|E| \choose k}$. Really, in this case, any subset consisting of $k$ elements of the set $|E|$ is necessarily independent in the matroid~$M^*$ and contains no broken circuit.

Since any cut in the complete graph $K_n$ consists of at least ${n-1}$ edges, the mentioned property implies Corollary~\ref{concl2}.

We can use the same property for finding coefficients of the characteristic polynomial of a uniform matroid or the dual one. Recall that the latter also is a uniform matroid.

Really, any cycle of a uniform matroid $U_{m,n}$, by definition, contains $m+1$ elements, and the power of its characteristic polynomial equals~$m$. Consequently, all coefficients of the characteristic polynomial, except, possibly, the free term, coincide (accurate to alternating signs) with the beginning of the $n$th row in the Pascal triangle.

Finally, for describing cocycles in $PG(n-1,q)$ we can make use of the fact that a cocycle is the minimal set of elements which has at least one common element with any base.
Hence, one can easily understand cocycles of the matroid $PG(n-1,q)$ geometrically. Namely, they represent sets of points in the projective space which belong to some hyperplane ${\mathbb F}_q^n$ (which does not contain the origin of coordinates). The number of points in any hyperplane ${\mathbb F}_q^n$ equals~$q^{n-1}$. Hence we get Corollary~\ref{concl5}.

\section{Conclusion}
In this paper, we propose a matroid form of one variant of the Matiyasevich formula and
describe its application for deriving flow polynomials
and characteristic polynomials of matroids dual to matroids
of finite projective geometries.
One can deduce most results from correlations for the corresponding Tutte polynomials. However, the use of one variable in the case, when two of them are not necessary, seems to give an easier way to obtain such formulas. On the other hand, this way allows us to better understand (in a sense) the interconnection of two variables in the Tutte polynomial. It also allows us to clarify the structure of major powers of considered polynomials.

In what follows, we are going to consider ``conditional'' flow and chromatic polynomials and their generalizations. In the simplest case, we are interested in the number of regular colorings of the graph, when colors of some vertices are fixed a~priori. For calculating this number, similarly to the conditional probability formula, we can write down the ratio of two chromatic polynomials, which explains the terminology used by us. It makes sense to consider the matroid variant of such correlations.

A month after the preprint of the first version of this paper was uploaded to arXiv.org, there has appeared preprint~\cite{ryo} by R.~Takashi, where he, in particular, develops ideas of this paper and gives more references. In~\cite[formula (3.2)]{reed}, R.~C.~Read and E.~G.~Whitehead re-derive the Matiyasevich formula. 
Moreover, in the same paper, they obtain an analog of this formula, which expresses a flow polynomial in terms of the sum of chromatic polynomials of factor graphs~\cite[formula (7.8)]{reed}. 
In~\cite{lob}, we prove the mentioned formula with the help of the Fourier transformations method (which is standard for Feynman amplitudes); in the Matiyasevich paper, this method is described regardless of Feynman amplitudes.
R.~C.~Read and E.~G.~Whitehead propose a more elementary proof; previously, in Section~1.2, we describe its simplified variant for the matroid case. 
As applied to flow and chromatic polynomials, the indicated formula has 4 variants (we have mentioned two of them earlier; in one more variant, the flow polynomial is represented as a linear combination of chromatic polynomials of subgraphs; the fourth variant of the formula expresses the chromatic polynomial in terms of flow polynomials of factor graphs). As we have mentioned earlier, the first formula among two ones indicated in braces is formula~(2.5) obtained in paper~\cite{kazakov}. R.~Takashi also gives all these formulas and their elementary proofs. Moreover, following~\cite{kazakov} (see also~\cite{sokal}), R.~Takashi for each formula considers its variant for a case, when each edge has a real-valued parameter, 
and generalizes these formulas for the matroid case. One can interpret this multiparameter variant as the partition function of multivariate Potts model .

Note that the multiparameter variants of these formulas, 
evidently, were obtained by D.~R.~Woodall in~\cite[Section 3]{woodall} with reference to Alan~Sokal. In addition, in the same paper, he also develops ideas of R.~C.~Read and E.~G.~Whitehead 
and gives detailed comments on them. In particular, regarding formulas~\eqref{one} and~\eqref{two} in the case of a graphic matroid and a bond one~\cite[formulas (1.1)--(1.4)]{woodall} 
(these formulas establish a connection between the flow polynomial~$F$ and the normalized chromatic polynomial denoted by the symbol~$N$), D.~R.~Woodall concludes that 
``The existence of expansions of this type is not surprising, but the interchange of $N$ and $F$ is intriguing and merits further study.'' 
One can consider this work as uncovering of this intrigue.

\subsubsection*{Acknowledgements} The research was supported by RSF (project No. 24-21-00158).


\begin{thebibliography}{99}
\bibitem{Andrews}
G.~E. Andrews, {\it The Theory of Partitions}, Cambridge Univ. Press, Cambridge, 1976.

\bibitem{reinePGen}
M.~Barany and V.~Reiner, {\it The Tutte polynomial of a finite projective
space}, (2005) preprint.

\bibitem{Bryl}
T.~Brylawski and J.~G.~Oxley, {\it The Tutte polynomial and its applications}, in Matroid Applications (ed. N. White), Encyclopedia of Mathematics and Its Applications, {\bf 40}, Cambridge
Univ. Press, 1992.

\bibitem{kazakov}
B.~Bychkov, A.~Kazakov, and D. Talalaev, {\it Functional Relations on Anisotropic Potts Models: from Biggs Formula
to the Tetrahedron Equation}, SIGMA {\bf 17}, (2021)  035. \doi{10.3842/SIGMA.2021.035}

\bibitem{reiner3}
W.~Kook, V.~Reiner\ and\ D.~Stanton, {\it A convolution formula for the Tutte polynomial}, J. Combin. Theory Ser. B {\bf 76} (1999), no.~2, 297--300. \MR{1699230}

\bibitem{kung2}
J.~P.~S.~Kung, {\it Convolution-multiplication identities for Tutte poly\-no\-mials of graphs and matroids}, J. Combin. Theory Ser. B {\bf 100} (2010), no.~6, 617--624. \MR{2718681}

\bibitem{ejc}
A.~P.~Kuptsov, E.~Yu.~Lerner, S.~A.~Mukhamedjanova, {\it Flow polynomials as Feynman amplitudes and their $\alpha$-representation}, Electron. J. Combin., {\bf 24} (2017), no.~1, paper 11, 19 pp.
\doi{10.37236/6396}

\bibitem{tmp}
E.~Yu.~Lerner, {\it Feynman integrals of $p$-adic argument in the momentum space. II. Explicit expressions},
Teoret. Mat. Fiz. {\bf 104} (1995), no. 3, 371--392; translation in Theoret. and Math. Phys. {\bf 104} (1995), no.~3, 1061--1077 (1996). \MR{1607005}

\bibitem{lerner2016}
E.~Yu.~Lerner, {\it The $\alpha$-representation for the characteristic function of a matroid},  arXiv:1611.02746, \doi{10.48550/arXiv.1611.02746}
(2016)

\bibitem{lob}
E. Yu.~Lerner and  S. A.~Mukhamedjanova,
{\it Matiyasevich Formula for Chromatic and Flow Polynomials and Feynman Amplitudes},
Lobachevskii Journal of Mathematics {\bf 43} (2022), no. 12, 3552--3561 \MR{4567102}

\bibitem{mat}
Yu. V.~Matiyasevich, {\it On a certain representation of the chromatic
po\-ly\-nomial}, Diskretnyi Analiz {\bf 91}, issue 31,
61--70 (1977). \MR{543806}, \href{https://arxiv.org/abs/0903.1213}{arXiv:0903.1213}

\bibitem{mexico}
C.~Merino, M.~Ramirez-Iba{\~n}ez, G. Rodriguez-S{\`a}nchez {\it The Tutte po\-ly\-no\-mial of some matroids},
{\bf 2012} (2012) Article ID 430859 \doi{10.1155/2012/430859}

\bibitem{MPHAKO}
E.~G.~Mphako. {\it Tutte Polynomials of Perfect Matroid Designs}, Combinatorics, Probability and Computing {bf 9}  (2000) 363–367
\doi{10.1017/S0963548300004338}

\bibitem{oxley}
J.~G.~Oxley, {\it Matroid Theory}, Oxford University Press, Oxford, New York, 1992. \MR{1207587}

\bibitem{reiner2}
V. Reiner, {\it An interpretation for the Tutte polynomial}, European J. Combin. {\bf 20} (1999), no.~2,
149--161 \MR{1676189}

\bibitem{reed}
R.~C.~ Read and E.~G.~ Whitehead Jr., {\it Chromatic polynomials of homeomorphism classes of graphs}, Discrete Mathematics {\bf 204} (1999), 
no.~1--3, 337--356 \doi{10.1016/S0012-365X(98)00378-1}

\bibitem{rota}
G.~C.~Rota  {\it On the foundations of combinatorial theory I. Theory of M\"{o}bius Functions}. Z. Wahrscheinlichkeitstheorie verw Gebiete {\bf 2} (1964) 340–-368.
\doi{10.1007/BF00531932}

\bibitem{sokal}
A. D.~Sokal, {\it The multivariate Tutte polynomial (alias Potts model) for graphs and
matroids}, Surv. Comb. {\bf 327}  (2005), 173--226.
\doi{10.1017/CBO9780511734885.009}

\bibitem{ryo}
R.~Takahashi, {\it Expansions of the Potts model partition function along deletions and contractions}, 	arXiv:2405.07612 (2024). 
\doi{10.48550/arXiv.2405.07612}

\bibitem{tutte1967}
W.~T.~Tutte, {\it On Dichromatic Polynomials}, Journal of Combinatorial Theory, Series A, {\bf 2} (1967), 301--320. \doi{10.1016/S0021-9800(67)80032-2}

\bibitem{tutte-book}
W.~T.~Tutte, {\it Graph theory}, reprint of the 1984 original, Encyclopedia of Mathematics and its Applications, 21, Cambridge Univ. Press, Cambridge, 2001. \MR{1813436}

\bibitem{welsh}
D.~Welsh, The Tutte polynomial, Random Structures Algorithms {\bf 15} (1999), no.~3-4, 210--228. \MR{1716762}

\bibitem{woodall}
D.~R.~Woodall, {\it Tutte polynomial expansions for 2-separable graphs},  Discrete Mathematics {\bf 247} (2002), no. 1, 201--213.
\doi{10.1016/S0012-365X(01)00177-7}

\bibitem{zasl}
T. Zaslavsky, {\it The M\"{o}bius Function and the Characteristic Polynomial},
Combinatorial geometries, Encyclopedia Math. Appl. {\bf 29} (1987) 114-–138 \MR{921071}

\end{thebibliography}
\end{document}